\documentclass{amsart}
\calclayout
 \usepackage{amssymb,amsmath,amsfonts,epsfig,latexsym,tikz}
   \usepackage[alphabetic]{amsrefs}
 \usepackage{tikz-cd}
\usepackage{hyperref}
\usepackage{enumerate}
\usepackage{mathtools}
\usepackage{verbatim}
\usepackage{cleveref}
\usepackage[shortlabels]{enumitem}
\usepackage{subcaption}

\usetikzlibrary{positioning}
\usetikzlibrary{matrix}
\usetikzlibrary{decorations}
\usetikzlibrary{decorations.pathreplacing, decorations.pathmorphing, angles,quotes}
 
 \newtheorem{theorem}{Theorem}[section]

\newtheorem{conjecture}[theorem]{Conjecture}

\theoremstyle{definition}
\newtheorem{remark}[theorem]{Remark}
\newtheorem{example}[theorem]{Example}

\begin{document}

\title{Counterexamples to two conjectures about matroids}          
    
\author{Matt Larson}
\begin{abstract}
We give counterexamples to two well-known conjectures about matroids: White's conjecture on the generation of the toric ideal by symmetric exchange binomials, and a conjecture of Mason on the log-concavity of the counts of flats of a given rank. 
\end{abstract}
 
\maketitle

\vspace{-20 pt}

In this paper, we describe counterexamples to two conjectures about matroids, White's conjecture that the symmetric exchange binomials generate the toric ideal of a matroid (Conjecture~\ref{conj:white}(\ref{whitesymmetric})) and a conjecture of Mason that counts of flats of a given rank in a matroid form a log-concave sequence (Conjecture~\ref{conj:logconc}). Our proofs do not require the use of a computer. We also recount the history of these conjectures and related results. 

One feature of these conjectures is that, for matroids of reasonable size, they can be checked easily using a computer. There are many known methods for producing interesting matroids. Until recently, it was tedious to test if these methods give counterexamples, but frontier large language models are capable of implementing these methods and checking if they give counterexamples. I have included detailed descriptions of how the counterexamples in this paper were found.

\subsection*{Acknowledgments}
This work was conducted while I was at the Institute for Advanced Study, where I am supported by the Charles Simonyi Endowment and the Oswald Veblen Fund. I am grateful to Krist\'of B\'erczi, Luis Ferroni, Mateusz Micha{\l}ek,  and Tam\'as Schwarcz for helpful comments on a previous draft. 
The use of ChatGPT 5.5 Pro in finding these counterexamples is described in Sections~\ref{sec:white} and \ref{sec:logconc}. I found the proofs given below and wrote the exposition. ChatGPT 5.5 Pro and Claude Opus 4.8 were also used for proofreading and for generating the figures.

\section{White's conjecture}\label{sec:white}

Given bases $B_1$ and $B_2$ of a matroid $\mathrm{M}$, a \emph{symmetric exchange} of the ordered pair $(B_1, B_2)$ is an ordered pair of bases $(\tilde{B}_1, \tilde{B}_2)$ with the property that there are $x \in B_1$ and $y \in B_2$ such that $\tilde{B}_1 = (B_1 \setminus x) \cup y$ and $\tilde{B}_2 = (B_2 \setminus y) \cup x$. It is known that for any pair $(B_1, B_2)$ and $x \in B_1 \setminus B_2$, there is some $y \in B_2 \setminus B_1$ such that $(B_1 \setminus x) \cup y$ and $(B_2 \setminus y) \cup x$ are both bases. In \cite{WhiteBases}, White made three conjectures about basis exchange properties of matroids. These conjectures are in decreasing order of strength. 

\begin{conjecture}\cite[Conjecture 12]{WhiteBases}\label{conj:white}
Let $\mathrm{M}$ be a matroid, and let $(B_1, \dotsc, B_s)$ and $(B_1', \dotsc, B_s')$ be sequences of bases of $\mathrm{M}$ such that the multiset union of $B_1, \dotsc, B_s$ is equal to the multiset union of $B_1', \dotsc, B_s'$. Then
\begin{enumerate}
\item \label{whitenc} It is possible to obtain $(B_1', \dotsc, B_s')$ from $(B_1, \dotsc, B_s)$ by repeatedly choosing two bases $B_i$ and $B_j$ with $i < j$, choosing a symmetric exchange $(\tilde{B}_i, \tilde{B}_j)$ of $(B_i, B_j)$, and replacing $(B_1, \dotsc, B_s)$ by $(B_1, \dotsc, B_{i-1}, \tilde{B}_i, \dotsc, B_{j-1}, \tilde{B}_j, \dotsc, B_s)$. 
\item \label{whitesymmetric} It is possible to obtain the multiset $\{B_1', \dotsc, B_s'\}$ from the multiset $\{B_1, \dotsc, B_s\}$ by repeatedly choosing two bases $B_i$ and $B_j$, choosing a symmetric exchange $(\tilde{B}_i, \tilde{B}_j)$, and replacing $\{B_1, \dotsc, B_s\}$ by $\{B_1, \dotsc, B_{i-1}, \tilde{B}_i, \dotsc, B_{j-1}, \tilde{B}_j, \dotsc, B_s\}$. 
\item \label{whitequadratic} It is possible to obtain the multiset $\{B_1', \dotsc, B_s'\}$ from the multiset $\{B_1, \dotsc, B_s\}$ by repeatedly choosing two bases $B_i$ and $B_j$, choosing any pair of bases $(\overline{B}_i, \overline{B}_j)$ with the property that the multiset union of $B_i$ and  $B_j$ is equal to the multiset union of $\overline{B}_i$ and $\overline{B}_j$, and replacing $\{B_1, \dotsc, B_s\}$ by the multiset $\{B_1, \dotsc, B_{i-1}, \overline{B}_i, \dotsc, B_{j-1}, \overline{B}_j, \dotsc, B_s\}$. 
\end{enumerate}
\end{conjecture}

In Example~\ref{ex:whitecounterexample}, we will give a counterexample to Conjecture~\ref{conj:white}(\ref{whitesymmetric}). We will first discuss the history of Conjecture~\ref{conj:white} and related results. 

Clearly Conjecture~\ref{conj:white}(\ref{whitesymmetric}) holds for a matroid $\mathrm{M}$ if and only if Conjecture~\ref{conj:white}(\ref{whitequadratic}) holds for $\mathrm{M}$ and Conjecture~\ref{conj:white}(\ref{whitesymmetric}) holds for $\mathrm{M}$ when $s=2$. We will be interested in the case of Conjecture~\ref{conj:white}(\ref{whitesymmetric}) when $s=2$, which is sometimes called Farber's conjecture after \cites{Farber1,Farber2}.

\begin{conjecture}\label{conj:farber}
Let $\mathrm{M}$ be a matroid, and let $\{B_1, B_2\}$ and $\{B_1', B_2'\}$ be unordered pairs of bases such that the multiset union of $B_1$ and $B_2$ is equal to the multiset union of $B_1'$ and $B_2'$. Then it is possible to obtain  $\{B_1', B_2'\}$ from $\{B_1, B_2\}$ by repeatedly replacing $\{B_1, B_2\}$ by a symmetric exchange $\{\tilde{B}_1, \tilde{B}_2\}$. 
\end{conjecture}

We will give a counterexample to Conjecture~\ref{conj:farber} in Example~\ref{ex:whitecounterexample}.

Clearly Conjecture~\ref{conj:white}(\ref{whitenc}) holds for a matroid $\mathrm{M}$ if and only if Conjecture~\ref{conj:white}(\ref{whitesymmetric}) holds for $\mathrm{M}$ and the following strengthening of Conjecture~\ref{conj:farber} holds for $\mathrm{M}$: for any ordered pair of bases $(B_1, B_2)$, it is possible to obtain $(B_2, B_1)$ by a sequence of symmetric exchanges. As observed in \cite[Section 4]{ToricIdeal}, this holds if Conjecture~\ref{conj:farber} holds for $\mathrm{M} \oplus \mathrm{M}$. If Conjecture~\ref{conj:farber} holds for $\mathrm{M} \oplus \mathrm{M}$, then the three forms of Conjecture~\ref{conj:white} for $\mathrm{M}$ are equivalent. 
A different way of reducing Conjecture~\ref{conj:white}(\ref{whitenc}) to Conjecture~\ref{conj:white}(\ref{whitesymmetric}) is given in \cite[Theorem 2.5]{BoninSparse}.

\smallskip

Conjecture~\ref{conj:white} can be interpreted in terms of the toric ideal of the matroid, the ideal of the homogeneous coordinate ring of the toric variety corresponding to the matroid polytope $P(\mathrm{M})$. If the ground set of $\mathrm{M}$ is $\{1, \dotsc, n\}$, this ideal can be described as the kernel of the map $k[x_B : B \text{ basis}] \to k[x_1, \dotsc, x_n]$ which sends $x_B$ to $\prod_{i \in B} x_i$. 
Conjecture~\ref{conj:white}(\ref{whitequadratic}) states that this ideal is generated by quadrics and Conjecture~\ref{conj:white}(\ref{whitesymmetric}) states that this ideal is generated by quadrics corresponding to symmetric exchange binomials. Conjecture~\ref{conj:white}(\ref{whitenc}) can be interpreted as giving generators for a version of the toric ideal in a non-commutative polynomial ring.

In \cite{HerzogHibi}, Herzog and Hibi conjectured an extension of Conjecture~\ref{conj:white}(\ref{whitesymmetric}) to polymatroids, and they asked whether the homogeneous coordinate ring of the toric variety of $P(\mathrm{M})$ is Koszul, and whether the toric ideal satisfies the stronger property of having a quadratic Gr\"{o}bner basis. By work of Sturmfels \cite[Theorem 8.3, Corollary 8.9]{SturmfelsGrobner}, it is known that the toric ideal has a quadratic Gr\"{o}bner basis if and only if the polytope $P(\mathrm{M})$ has a regular unimodular triangulation whose underlying simplicial complex is flag, i.e., all minimal non-faces have size $2$. A regular unimodular triangulation of $P(\mathrm{M})$ was constructed in \cite{BackmanLiu}, but it was recently shown that the matroid polytope of the Fano matroid does not have a flag regular unimodular triangulation \cite{Grobnerwhite1,Grobnerwhite2}.

Note that, in Conjecture~\ref{conj:farber}, we can immediately reduce to the case when $B_1 \cap B_2 = \emptyset$ and $B_1 \cup B_2$ is the ground set of $\mathrm{M}$ by deleting elements not contained in $B_1 \cup B_2$ and contracting elements in $B_1 \cap B_2$. This means that for  matroids of rank at most $r$, Conjecture~\ref{conj:farber} reduces to the case of matroids on a ground set of size at most $2r$. This is a finite computation, and one can check that Conjecture~\ref{conj:farber} holds for matroids of rank at most $4$. 

\smallskip

There is a very extensive literature proving White's conjectures or their variants for various families of matroids. We mention some of the most relevant results. 
Conjecture~\ref{conj:white}(\ref{whitesymmetric}) was proven for graphic matroids by Blasiak \cite{Blasiak} (after Conjecture~\ref{conj:farber} was proven in \cite{Farber1} for graphic matroids); this result was extended by McGuinness \cite{McGuinness}. 
Conjecture~\ref{conj:white}(\ref{whitesymmetric}) was proven for transversal matroids in \cite{ToricIdeal} (after Conjecture~\ref{conj:farber} was proven in \cite{Farber2} for transversal matroids). The existence of a quadratic Gr\"{o}bner basis for the toric ideal was shown for lattice path matroids, a family of transversal matroids, in \cite{Schweig}; this was earlier essentially proved by Knudsen \cite[Chapter III, Lemma 2.4]{KKMS}, see \cite[Theorem 3.3]{HPPS}. 
Conjecture~\ref{conj:white}(\ref{whitesymmetric}) was proven for sparse paving matroids in \cite{BoninSparse} and then for paving matroids in \cite{YuYuen}, following ideas introduced in \cite{innerproj}. 
Conjecture~\ref{conj:white}(\ref{whitesymmetric}) was proven for matroids of rank $3$ by Kashiwabara \cite{Kashiwabara}. 
Conjecture~\ref{conj:farber} was proven for split matroids in \cite{SplitExchange} and for regular matroids in \cite{RegularExchange}. 

It was shown that the symmetric exchange binomials generate the toric ideal up to saturation in \cite{ToricIdeal}. 
In \cite{LasonToric}, it was shown that the toric ideal of a matroid of rank $r$ has a Gr\"{o}bner basis consisting of polynomials of degree at most $(r+3)!$, and Conjecture~\ref{conj:white}(\ref{whitesymmetric}) was proven if $s$ is larger than a certain function of $r$. 
In particular, this reduces checking any version of Conjecture~\ref{conj:white} to a finite computation.

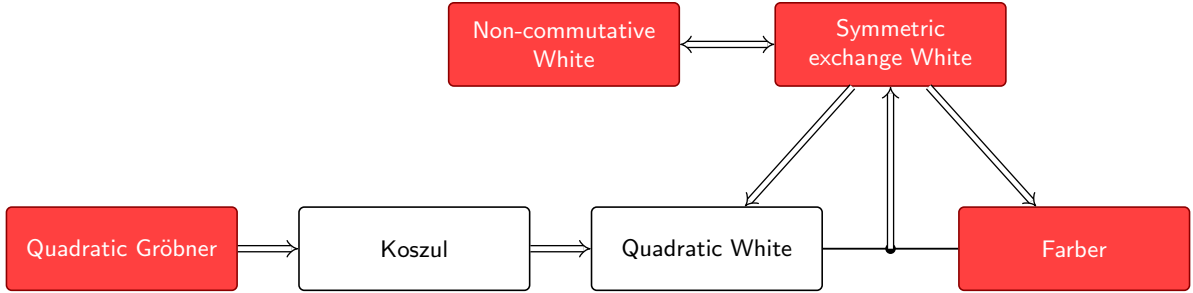
\begin{figure}
\begin{tikzpicture}[scale = 0.9,
  box/.style={
    draw=black, rounded corners=2pt,
    text width=2.7cm, align=center,
    minimum height=1.1cm, inner sep=5pt,
    font=\sffamily\small, line width=0.6pt
  },
  redbox/.style={box, fill=red!75, text=white, draw=red!55!black},
  whitebox/.style={box, fill=white, text=black},
  impl/.style={double, double distance=1.8pt, line width=0.5pt,
               -{Implies[length=2.6mm,width=3.4mm]}},
  biimpl/.style={double, double distance=1.8pt, line width=0.5pt,
                 {Implies[length=2.6mm,width=3.4mm]}-{Implies[length=2.6mm,width=3.4mm]}},
  gather/.style={line width=0.7pt}
]
 
\node[redbox]   (qg)  at (0, 0)    {Quadratic Gr\"{o}bner};
\node[whitebox] (kos) at (4.3, 0)  {Koszul};
\node[whitebox] (qw)  at (8.6, 0)  {Quadratic White};
\node[redbox]   (far) at (14, 0)   {Farber};
\node[redbox]   (se)  at (11.3, 3) {Symmetric exchange White};
\node[redbox]   (nc)  at (6.5, 3)  {Non-commutative White};
 
\draw[impl] (qg)  -- (kos);
\draw[impl] (kos) -- (qw);
 
\draw[biimpl] (nc) -- (se);

\draw[impl] (se) -- (qw);
\draw[impl] (se) -- (far);
 
\coordinate (jct) at (11.3, 0);
\draw[gather] (qw)  -- (jct);
\draw[gather] (far) -- (jct);
\fill (jct) circle (2.2pt);
\draw[impl] (jct) -- (se.south);

\end{tikzpicture}
\caption{Implications between some variants of White's conjecture. The conjectures which are known to be false are colored red.}
\end{figure}

\begin{remark}
We are not aware of any family of matroids where Conjecture~\ref{conj:white}(\ref{whitequadratic}) is known, but Conjecture~\ref{conj:white}(\ref{whitesymmetric}) is open, or where the homogeneous coordinate ring of the toric variety of $P(\mathrm{M})$ is known to be Koszul, but its ideal is not known to have a quadratic Gr\"{o}bner basis.  However, \cite{MoriOhsugi} gives some evidence for Conjecture~\ref{conj:white}(\ref{whitequadratic}), and Conca proves that the homogeneous coordinate rings of transversal polymatroids are Koszul in \cite{ConcaWhite}.
\end{remark}

\begin{example}\label{ex:whitecounterexample}
Let $\mathrm{M}$ be the rank $9$ binary matroid on $\{1, \dotsc, 18\}$ realized by the row span of the following matrix 
$$\left(\begin{array}{cccccc|cccccc|cccccc}
0 & 0 & 1 & 1 & 0 & 0 & 0 & 1 & 0 & 1 & 0 & 1 & 0 & 0 & 0 & 0 & 0 & 0 \\
0 & 1 & 0 & 1 & 0 & 1 & 0 & 0 & 0 & 0 & 0 & 0 & 1 & 1 & 0 & 0 & 0 & 0 \\
0 & 0 & 0 & 0 & 0 & 0 & 1 & 1 & 0 & 0 & 0 & 0 & 0 & 1 & 0 & 1 & 0 & 1 \\
\hline
1 & 1 & 1 & 1 & 0 & 0 & 0 & 0 & 0 & 0 & 0 & 0 & 0 & 0 & 0 & 0 & 0 & 0 \\
0 & 0 & 1 & 1 & 1 & 1 & 0 & 0 & 0 & 0 & 0 & 0 & 0 & 0 & 0 & 0 & 0 & 0 \\
\hline
0 & 0 & 0 & 0 & 0 & 0 & 1 & 1 & 1 & 1 & 0 & 0 & 0 & 0 & 0 & 0 & 0 & 0 \\
0 & 0 & 0 & 0 & 0 & 0 & 1 & 1 & 0 & 0 & 1 & 1 & 0 & 0 & 0 & 0 & 0 & 0 \\
\hline
0 & 0 & 0 & 0 & 0 & 0 & 0 & 0 & 0 & 0 & 0 & 0 & 1 & 1 & 1 & 1 & 0 & 0 \\
0 & 0 & 0 & 0 & 0 & 0 & 0 & 0 & 0 & 0 & 0 & 0 & 1 & 1 & 0 & 0 & 1 & 1
\end{array}\right);$$
the vertical and horizontal lines have been added to make the block structure more visible. 
\end{example}

\begin{theorem}
Conjecture~\ref{conj:farber} and Conjecture~\ref{conj:white}(\ref{whitesymmetric}) fail in Example~\ref{ex:whitecounterexample}. 
\end{theorem}

\begin{proof}
The ground set $\{1, \dotsc, 18\}$ divides into three blocks: $\{1, \dotsc, 6\}$, $\{7, \dotsc, 12\}$, and $\{13, \dotsc, 18\}$. It is not hard to check that the element of $S_{18}$ which cyclically permutes these blocks (sending $1$ to $7$, $7$ to $13$, $13$ to $1$, $2$ to $8$, and so on) is an automorphism of this matroid. The restriction of $\mathrm{M}$ to each block is a matroid of rank $4$ which has three coparallel classes of size $2$. This divides the ground set into $9$ pairs of the form $\{2i - 1, 2i\}$ for $i = 1, \dotsc, 9$. There is an action of $S_4$ on each block, where the action of $S_4$ on a block is the same as its action on the edges of the complete graph $\mathrm{K}_4$, and the pairs correspond to vertex-disjoint edges. 
These elements generate a subgroup of the automorphism group of $\mathrm{M}$ isomorphic to $(S_4)^3 \rtimes \mathbb{Z}/3 \mathbb{Z}$.

Let $B$ be a basis of $\mathrm{M}$ such that $B^c$ is also a basis. Let $s(B)$ be the number of pairs $\{2i - 1, 2i\}$ which are split by $B$, i.e., $|B \cap \{2i-1, 2i\}| = 1$. Because the rank of $\mathrm{M}$ is odd, $s(B)$ must be odd. Note that $s(B) = s(B^c)$, and that if $\tilde{B}$ differs from $B$ by a single symmetric exchange, then $s(B) - s(\tilde{B}) \in \{-2, 0, 2\}$. We claim that $s(B)$ lies in $\{3, 7, 9\}$. All three values occur: we have
$$s(\{1, 2, 3, 7, 8, 9, 13, 14, 15\}) = 3, \text{  } s(\{1, 2, 3, 5, 7, 9, 11, 13, 15\}) = 7, \text{ and }s(\{1, 3, 5, 7, 9, 11, 13, 15, 18\}) = 9,$$
and each of the sets above is a basis whose complement is also a basis. 
The pairs $\{B, B^c\}$ with $s(B) = 3$ cannot be transformed into a pair $\{B', (B')^c\}$ with $s(B') \in \{7, 9\}$, showing that Conjecture~\ref{conj:farber} and Conjecture~\ref{conj:white}(\ref{whitesymmetric}) fail for $\mathrm{M}$. 

Because $B$ and $B^c$ are both bases and the rank of each block is $4$, the intersection of $B$ with any block has either $2, 3$, or $4$ elements. The intersection of $B$ with the three blocks either gives sets of size $2, 3,$ and $4$ or sets of size $3, 3,$ and $3$. The union of any two pairs in a block is a circuit, so an independent set of size $4$ in a block must split $2$ pairs. This implies that $s(B) \not=1$. 
Because both $B$ and $B^c$ are bases, an independent set of size $2$ in a block must split $2$ pairs. If the intersection of $B$ with a block is a set of size $4$, then, up to the action of the automorphisms, the set must be $\{1, 2, 3, 5\}$. If the intersection has size $2$, then, up to the action of automorphisms, the set must be $\{1, 3\}$. If the intersection has size $3$, then there are three possibilities: we could get $\{1, 3, 5\}, \{1, 3, 6\}$, or $\{1, 2, 3\}$; there are three orbits of subgraphs of $\mathrm{K}_4$ with $3$ edges.  These split $3, 3$, and $1$ pairs, respectively. 

We consider sets of size $9$ which split exactly $5$ pairs and satisfy the conditions in the previous paragraph, up to the action of automorphisms. There are only four: representatives are
$$\{1, 2, 3, 7, 8, 9, 13, 15, 18\}, \text{ }\{1, 2, 3, 7, 8, 9, 13, 15, 17\}, \text{  } \{1, 2, 3, 5, 7, 9, 13, 14, 15\}, \text{ and }\{1, 2, 3, 5, 7, 8, 9, 13, 15\}.$$
The complement of a set of the first type is a set of the second type, and the complement of a set of the third type is a set of the fourth type. The first and third sets are bases, but the second and fourth sets are not, so none of these gives rise to pairs of bases $\{B, B^c\}$ with $s(B) = 5$. 
\end{proof}

\subsection*{How Example~\ref{ex:whitecounterexample} was found}

After checking that there are no counterexamples to Conjecture~\ref{conj:farber} in rank at most $4$, I prompted ChatGPT 5.5 Pro to try to check Conjecture~\ref{conj:farber} in small rank using a SAT solver. It implemented an approach and claims to have used this to prove Conjecture~\ref{conj:farber} in rank $5$ and $6$\footnote{I have not verified the correctness of this computation.}. I checked Conjecture~\ref{conj:farber} for many random matroids of ranks $7, 8, 9$, and $10$ which are realizable over $\mathbb{F}_2, \mathbb{F}_3, \mathbb{F}_4$, or $\mathbb{F}_5$. I prompted ChatGPT to look for a counterexample around 20 times; see \href{https://chatgpt.com/share/6a27edd1-c608-83ea-95f7-c2b79b2a88ac}{here} for a typical chat. I then prompted it to look for ternary matroids of ranks $7, 8, 9$, and $10$ and then to look for binary matroids of ranks $7$ and $8$. When I prompted it to look at binary matroids of rank $9$, it found a matroid isomorphic to the matroid in Example~\ref{ex:whitecounterexample}; see \href{https://chatgpt.com/share/6a27ed96-0374-83ea-8daf-8057d7d6ceb5}{here}. 

Another description of the matroid in Example~\ref{ex:whitecounterexample} is that it is isomorphic to the binary matroid realized by the row span of the $9 \times 18$ matrix $[I|A]$, where $I$ is the $9 \times 9$ identity matrix and 
$$A = \begin{pmatrix} 1 & 0 & 0 & 1 & 0 & 0 & 1 & 0 & 1\\ 
1 & 1 & 0 & 0 & 1 & 0 & 0 & 1 & 0 \\
0 & 1 & 1 & 0 & 0 & 1 & 0 & 0 & 1  \\
1 & 0 & 1 & 1 & 0 & 0 & 1 & 0 & 0 \\ 
0 & 1 & 0 & 1 & 1 & 0 & 0 & 1 & 0 \\ 
0 & 0 & 1 & 0 & 1 & 1 & 0 & 0 & 1 \\ 
1 & 0 & 0 & 1 & 0 & 1 & 1 & 0 & 0 \\ 
0 & 1 & 0 & 0 & 1 & 0 & 1 & 1 & 0 \\ 
0 & 0 & 1 & 0 & 0 & 1 & 0 & 1 & 1
\end{pmatrix}.$$
In \cite[Example 15]{WhiteBases}, White considered the family of binary matroids represented by matrices of the form $[I|A]$ where $A$ is a circulant matrix whose nonzero entries in the first row are consecutive. Subspaces generated by the row span of a matrix $[I|A]$, where $A$ is a circulant matrix, are called \emph{quasi-cyclic codes} in the coding theory literature; see \cite{QuasiCyclic}.

We have been unable to find a counterexample to Conjecture~\ref{conj:white}(\ref{whitequadratic}). One issue is that it is computationally difficult to check Conjecture~\ref{conj:white}(\ref{whitequadratic}). On a laptop with an Intel i7-1250U chip, checking that Conjecture~\ref{conj:farber} fails in Example~\ref{ex:whitecounterexample}  takes less than half a second, but checking Conjecture~\ref{conj:white}(\ref{whitequadratic}) for a single matroid of rank $6$ on $12$ elements is already difficult.  Example~\ref{ex:whitecounterexample} and similar constructions do not seem to give a counterexample to other well-known conjectures about basis exchange, such as Gabow's conjecture \cite{Gabow}.

\section{Log-concavity of counts of flats}\label{sec:logconc}

For a matroid $\mathrm{M}$, let $W_i$ denote the number of flats of $\mathrm{M}$ of rank $i$. These are known as the Whitney numbers of the second kind. 

\begin{conjecture}\label{conj:logconc}
Let $\mathrm{M}$ be a matroid of rank $r$. Then for any $1 \le i \le r-1$, we have $W_i^2 \ge W_{i+1} W_{i-1}$. 
\end{conjecture}

This was conjectured by Mason in \cite[Conjecture 1B]{Mason}, where he also conjectured two stronger versions predicting the ultra log-concavity of the $W_i$. The weaker statement that the $W_i$ are unimodal was conjectured by Rota in \cite{RotaICM}. Conjecture~\ref{conj:logconc} is Problem 25(c) in Stanley's list of positivity problems and conjectures in algebraic combinatorics \cite{StanleyPositivity}. There, he also mentions the log-concavity conjectures for the coefficients of the characteristic polynomial of a matroid and for the number of independent sets of a given size in a matroid. He speculates that all three of these conjectures are false. However, the log-concavity of the coefficients of the characteristic polynomial and the number of independent sets were proven in \cites{AHK18}. 

One case that has attracted particular attention is log-concavity at $W_2$: that $W_2^2 \ge W_1 W_3$. This is known as the points-lines-planes conjecture, although this is sometimes used to refer to the stronger conjectures that $W_2^2 \ge \frac{3}{2} W_1 W_3$ or that $W_2^2 \ge \frac{3}{2} \frac{W_1 - 1}{W_1 - 2} W_1 W_3$. This strongest form was proven by Seymour for matroids where each flat of rank $2$ has size at most four \cite{Seymour}. In particular, this applies to all binary matroids. This case was previously proven for graphic matroids in \cite{Stonesifer}, and it was subsequently studied by Dukes \cite{Dukes}. 

In contrast to Conjecture~\ref{conj:white}, there are very few families of matroids for which Conjecture~\ref{conj:logconc} is known. It holds trivially for paving matroids, and it is easy to check for matroids associated to the complete graph or projective geometries. See \cite{Aigner} for a survey and a few additional cases. 

There are two general results which can be interpreted as giving evidence for Conjecture~\ref{conj:logconc}. The first is the log-concavity of the coefficients of the characteristic polynomial of a matroid, proved in \cite{Huh2012,HK12} for realizable matroids and in \cite{AHK18} for all matroids. The coefficients of the characteristic polynomial can be interpreted as counting the number of flats of a given rank, except that we weight a flat $F$ by the M\"{o}bius value $|\mu(\emptyset, F)|$. The second is the Dowling--Wilson top-heavy conjecture \cite{DW1,DW2}, which states that for a matroid of rank $r$, we have $W_0 \le W_1 \le \dotsb \le W_{\lfloor r/2 \rfloor}$, and $W_i \le W_{r-i}$ for $i \le r/2$. This was proven in \cite{HW} for realizable matroids and in \cite{BHMPW20b} for all matroids. The proof was simplified in \cite{LefschetzModule}. 

\def\Len{10}      
\def\dotr{1.5pt} 
\def\bigr{3pt}    
 
\tikzset{
  edge/.style={line width=.8pt, line cap=round, line join=round},
  vert/.style={fill=black},
  term/.style={fill=black},
}
 
\newcommand{\longpath}[2]{
  \pgfmathsetmacro{\Rr}{(\Len*\Len)/(8*(#1)) + (#1)/2}
  \pgfmathsetmacro{\Cy}{(#1) - \Rr}
  \pgfmathsetmacro{\angv}{atan2(\Rr-(#1), -\Len/2)}
  \pgfmathsetmacro{\angw}{atan2(\Rr-(#1),  \Len/2)}
  \draw[edge] (0,0)
    \foreach \i in {1,...,26} {
      -- ({\Len/2 + \Rr*cos(\angv + (\i/26)*(\angw-\angv))},{(#2)*(\Cy + \Rr*sin(\angv + (\i/26)*(\angw-\angv)))})
    };
  \foreach \i in {1,...,25} {
    \fill[vert] ({\Len/2 + \Rr*cos(\angv + (\i/26)*(\angw-\angv))},{(#2)*(\Cy + \Rr*sin(\angv + (\i/26)*(\angw-\angv)))}) circle (\dotr);
  }
}

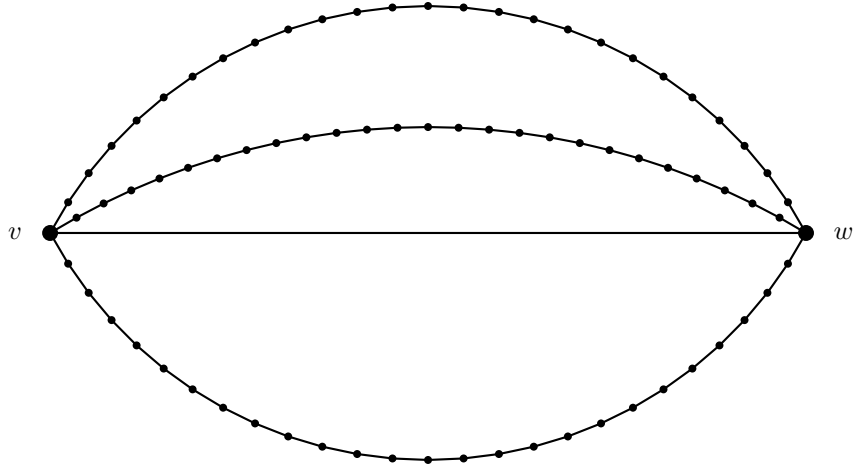
\begin{figure}
\begin{tikzpicture}
  \draw[edge] (0,0) -- (\Len,0);
 
  \longpath{3.0}{ 1}   
  \longpath{1.4}{ 1}   
  \longpath{3.0}{-1}   
 
  \fill[term] (0,0)    circle (\bigr);
  \fill[term] (\Len,0) circle (\bigr);
  \node[left=7pt]  at (0,0)    {$v$};
  \node[right=7pt] at (\Len,0) {$w$};
\end{tikzpicture}
\caption{The matroid in Example~\ref{ex:logconc}}\label{fig:logconc}
\end{figure}

\begin{example}\label{ex:logconc}
Let $\mathrm{M}$ be the graphic matroid corresponding to the graph consisting of two vertices $\{v, w\}$ and four paths between them, with three of the paths consisting of $26$ \emph{edges} and the other path consisting of a single edge. See Figure~\ref{fig:logconc}. 
\end{example}

\begin{theorem}
Conjecture~\ref{conj:logconc} fails in Example~\ref{ex:logconc}. 
\end{theorem}

\begin{proof}
The dual of $\mathrm{M}$ is a matroid $\mathrm{N}$ of rank $3$ on a ground set of size $79$ whose simplification is isomorphic to the uniform matroid $\mathrm{U}_{3,4}$, and which has parallel classes of size $1, 26, 26,$ and $26$. The number of flats of $\mathrm{M}$ of rank $76 - i$ is equal to the number of cyclic sets (union of circuits) $S$ in $\mathrm{N}$ with $|S| - \operatorname{rk}_N(S) = i$. 

A cyclic set of $\mathrm{N}$ with $|S| - \operatorname{rk}_N(S) = 1$ consists of either two elements of one of the large parallel classes or it intersects all three large parallel classes in one element and uses the edge in the path of length $1$, giving 
$$W_{75} =  3 \binom{26}{2} + 26^3 = 18,551.$$
A cyclic set of $\mathrm{N}$ with $|S| - \operatorname{rk}_N(S) = 2$ consists of either three elements of one large parallel class; two pairs from two large parallel classes; or two elements from one large parallel class, one from each of the other two and the edge in the path of length $1$. This gives 
$$W_{74} = 3 \binom{26}{3} + 3 \binom{26}{2}^2 + 3 \binom{26}{2} \cdot 26^2 = 983,775.$$
A cyclic set of $\mathrm{N}$ with $|S| - \operatorname{rk}_N(S) = 3$ consists of either four elements from one large parallel class; three elements from one large parallel class and two from another; or intersects all three large parallel classes, which can happen in three different ways. This gives
$$W_{73} = 3 \binom{26}{4} + 6 \binom{26}{3}\binom{26}{2} + \binom{26}{2}^3 + 3 \binom{26}{3} 26^2 + 3 \binom{26}{2}^2 26 = 52,954,525.$$
The result follows from the computation
\begin{equation*}W_{74}^2 = 967,813,250,625 < 982,359,393,275 =  W_{73} \cdot W_{75}. \qedhere\end{equation*}
\end{proof}

\subsection*{How Example~\ref{ex:logconc} was found}
I initially searched for a counterexample to the ultra log-concavity version of Conjecture~\ref{conj:logconc}. After checking that there are no counterexamples to this conjecture on at most $9$ elements, I prompted ChatGPT 5.5 Pro to look for a counterexample around 20 times. I then prompted it to look for counterexamples which are realizable over $\mathbb{F}_2$, $\mathbb{F}_3$, $\mathbb{F}_4$, and $\mathbb{F}_5$, and to consider matroids on large ground sets. After being instructed to look at matroids which are realizable over $\mathbb{F}_5$ and at failures of log-concavity at high indices, it found a variant of Example~\ref{ex:logconc}; see \href{https://chatgpt.com/share/6a30bef8-af94-83ea-babc-574f4598d841}{here}. 

The graph in Example~\ref{ex:logconc} is called a ``generalized theta graph.'' Generalized theta graphs were used by Sokal to show that the roots of characteristic polynomials of graphic matroids are dense in the complex plane \cite{Sokal}. 
Example~\ref{ex:logconc} is  similar to the example used in \cite{MeroniWelsh} to disprove the Merino--Welsh conjecture. There, the authors used a matroid which simplifies to a uniform matroid in which all of the parallel classes have size $2$. To disprove Conjecture~\ref{conj:logconc}, it appears to be advantageous to use parallel classes of very different sizes. 

I have been unable to find a counterexample to the unimodality of the $W_i$. 

\bibliography{matroid}
\bibstyle{amsalpha}

\end{document}